\documentclass[11pt]{article}
\usepackage{latexsym, amssymb, amsmath, amsfonts, amsthm}

\begin{document}

\theoremstyle{plain}
\newtheorem{prop}{Proposition}
\newtheorem{theorem}{Theorem}
\newtheorem{corollary}{Corollary}
\newtheorem{lemma}{Lemma}

\theoremstyle{definition}
\newtheorem{definition}{Definition}
\newtheorem{remark}{Remark}
\newtheorem{remarks}{Remarks}
\newtheorem{example}{Example}
\newtheorem{examples}{Examples}

\newcommand{\rr}{\mathbb R}
\newcommand{\nn}{\mathbb N}
\newcommand{\qq}{\mathbb Q}
\newcommand{\zz}{\mathbb Z}
\renewcommand{\ss}{\mathbb S}
\newcommand{\sse}{\subseteq}
\newcommand{\nti}{\not \in}
\newcommand{\spera}{\text{Sper } A}

\newcommand{\rp}{\rr^+[X]}
\newcommand{\ts}{\thinspace}

\newcommand{\al}{\alpha}
\newcommand{\be}{\beta}
\newcommand{\ga}{\gamma}
\newcommand{\de}{\delta}
\newcommand{\sig}{\sigma}
\newcommand{\Ga}{\Gamma}
\newcommand{\la}{\lambda}
\newcommand{\La}{\Lambda}
\newcommand{\all}{\frac{\alpha}{|\alpha|}}
\newcommand{\ep}{\epsilon}
\newcommand{\alo}{\alpha_1}
\newcommand{\ai}{\alpha_i}
\newcommand{\ak}{\alpha_k}
\newcommand{\an}{\alpha_n}
\newcommand{\supp}{{\rm supp}}

\title{Rational Certificates of Positivity on Compact Semialgebraic Sets}
\author{Victoria Powers \footnote{
Department of Mathematics and Computer Science,
Emory University,
Atlanta, GA 30322. Email:\, vicki@mathcs.emory.edu.}}

\maketitle

\begin{abstract}  Given $g_1, \dots , g_s \in \rr[X] = \rr[X_1, \dots , X_n]$ such that the
semialgebraic set $K := \{ x \in \rr^n \mid g_i(x) \geq 0$ for all $i \}$ is compact.
Schm\"udgen's Theorem says that if $f \in \rr[X]$ such that $f > 0$ on $K$, then
$f$ is in the preordering in $\rr[X]$ generated by the $g_i$'s, i.e., $f$ can be
written as a finite sum of elements $\sig g_1^{e_1} \dots g_s^{e_s}$, where $\sig$ is
a sum of squares in $\rr[X]$ and each $e_i \in \{0,1\}$.  Putinar's Theorem says
that under a condition stronger than compactness, any $f > 0$ on $K$ can be written
$f = \sig_0 + \sig_1 g_1 + \dots + \sig_s g_s$, where $\sig_i \in \rr[X]$.  Both of
these theorems can be viewed as statements about the existence of certificates of
positivity on compact semialgebraic sets.  In this note we show that if the defining
polynomials $g_1, \dots , g_s$ and polynomial $f$ have coefficients in $\qq$, then
in Schm\"udgen's Theorem we can find a representation in which the $\sig$'s are
sums of squares of polynomials over $\qq$.  We prove a similar result for Putinar's
Theorem assuming that the set of generators contains $N - \sum X_i^2$ for some
$N \in \nn$.
\end{abstract}

\section{Introduction}

We write $\nn$, $\rr$, and $\qq$ for the set of natural, real, and rational numbers.
Let $n \in \nn$ be fixed and
let $\rr[X]$ denote the polynomial
ring $\rr[X_1, \dots , X_n]$.   We denote by
$\sum \rr[X]^2$ the set of sums of squares in $\rr[X]$.

For $S = \{ g_1, \dots , g_s \} \sse \rr[X]$, the {\it basic closed
semialgebraic set} generated by $S$, denoted $K_S$, is
$$\{ x \in \rr^n \mid g_1(x) \geq 0, \dots , g_s(x) \geq 0 \}.$$
Associated to $S$ are two algebraic objects:  The {\it quadratic module
generated by $S$}, denoted $M_S$,  is the set of $f \in \rr[X]$
which can be written $$f = \sig_0 + \sig_1 g_1 + \dots + \sig_s g_s,$$
where each $\sig_i \in \sum \rr[X]^2$, and the {\it preordering generated
by $S$}, denoted $T_S$, is the quadratic module generated by all
products of elements in $S$.  In other words, $T_S$ is the set
of $f \in \rr[X]$ which can be written as a finite sum of
elements $\sig g_1^{e_1} \dots g_s^{e_s}$, where $\sig \in \rr[X]$ and
each $e_i \in \{ 0,1 \}$.

A polynomial $f \in \sum \rr[X]^2$ is obviously globally nonnegative
in $\rr^n$ and writing $f$ explicitly as a sum of squares gives a
``certificate of positivity" for the fact that $f$ takes only nonnegative
values in $\rr^n$.  (Note:  To avoid having to write ``nonnegative or
positive" we use the term ``positivity" to mean either.)  More generally,
for a basic closed semialgebraic set $K_S$, if $f \in T_S$ or $f \in M_S$,
then $f$ is nonnegative on $K_S$ and an explicit representation of
$f$ in $M_S$ or $T_S$ gives a certificate of
positivity for $f$ on $K_S$.

In 1991, Schm\"udgen \cite{schm} showed that if the semialgebraic
set $K_S$ is compact, then any $f \in \rr[X]$ which is strictly positive
on $K_S$ is in the preordering $T_S$.
In 1993, Putinar \cite{put} showed that under a certain condition which
is stronger than compactness any $f \in \rr[X]$ which is strictly
positive on $K_S$ is in the quadratic module $M_S$.  In other words, these
results say that under the given conditions a certificate of positivity
for $f$ on $K_S$ exists.

Recently, techniques from semidefinite
programming combined with Schm\"udgen's and Putinar's theorems have been used to give numerical
algorithms for applications such as
optimization of polynomials on semialgebraic sets.  However since these
algorithms are numerical they might not produce exact certificates
of positivity.
With this in mind, Sturmfels asked whether any $f \in \qq[X]$ which
is a sum of squares in $\rr[X]$ is a sum of squares in $\qq[X]$.  In \cite{hil},
Hillar showed that the answer is ``yes" in the case where $f$ is known to be
a sum of squares over a field $K$ which is Galois over $\qq$.  The general
question remains unsolved.

It is natural to ask a similar question for Schm\"udgen's
Theorem and Putinar's Theorem:  If the polynomials defining the semialgebraic
set and the positive polynomial $f$ have rational coefficients, is there a certificate
of positivity for $f$ in which the sums of squares have rational coefficients?
In this note, we show that in the case of Schm\"udgen's Theorem the answer is
``yes".  This follows from an algebraic proof of the theorem, originally due to
T.~W\"ormann \cite{wor}.  In the case of Putinar's Theorem, we show that the
answer is also ``yes" as long as the generating set contains $N - \sum X_i^2$
for some $N \in \nn$.  This follows easily from an algorithmic proof of the
theorem due to Schweighofer \cite{schw2}.  For Lasserre's method for optimization
of polynomials on compact semialgebraic sets, see \cite{LAS}, in concrete cases
it is usual to add a polynomial of the type $N - \sum X_i^2$ to the generators in order to
insure that Putinar's Theorem holds.  Thus our assumption in this case is
reasonable.

\section{Rational certificates of for Schm\"udgen's Theorem}

Fix $S = \{ g_1, \dots , g_s \} \sse \rr[X]$ and define $K_S$
and $T_S$ as above.

\begin{theorem}(Schm\"udgen) Suppose that $K_S$ is compact.  If $f \in \rr[X]$ and
$f > 0$ on $K_S$, then $f \in T_S$.
\end{theorem}

 In this section we show that
if  $f$ and the generating polynomials $g_1, \dots , g_s$
are in $\qq[X]$, then $f$ has a representation in $T_S$ in
which all sums of squares $\sigma_\ep$ are in $\sum \qq[X]^2$.
This follows  from T.~W\"ormann's algebraic proof
of the theorem using the classical Abstract Positivstellensatz,
and a generalization of W\"ormann's crucial lemma due
to M.~Schweighofer.

\medskip

\noindent
{\bf The Abstract Positivstellensatz}.
The Abstract
Positivstellensatz is usually attributed to Kadison-Dubois,
but now thought to be proven earlier by Krivine or Stone.  For
details on the history of the result, see \cite[Section 5.6]{PD}.  The setting is preordered
commutative rings.

Let $A$ be a commutative ring with $\qq \sse A$.  A subset $T \sse A$ is
a
{ \it preordering} if $T + T \sse T$, $T \cdot T \sse T$, and
$-1 \not \in T$.  For $S = \{ a_1, \dots , a_k \} \sse A$, we define
the {\it preordering generated by $S$}, $T_S$, exactly as for $A=\rr[X]$.

An {\it ordering} in $A$ is a preordering $P$ such that $P \cup -P = A$
and $P \cap -P$ is a prime ideal.  Any $a \in A$ has a unique sign in $\{ -1,0,1 \}$
with respect to a fixed ordering $P$ and we use the notation $a \geq_P 0$ if
$a \in P$, $a >_P 0$ if $a \in P \setminus (P \cap -P)$, etc.

Fix a preordered ring $(A,T)$ and denote by $\spera$ the real spectrum
of $(A,T)$, i.e., the set of orderings of $A$
which contain $T$.  Then define
$$H(A) = \{ a \in A \mid \text{ there exists }n \in \nn \text{ with } n \pm a \geq_P 0
\text{ for all } P \in \spera \},$$ the {\it ring of geometrically bounded elements in $(A,T)$}, and
$$H'(A) = \{ a \in A \mid \text{ there exists }n \in \nn \text{ with }n \pm a \in T  \},$$ the {\it ring of
arithmetically bounded elements in $(A,T)$}.  Clearly, $H'(A) \sse H(A)$.  The preordering $T$ is
{\it archimedean} if $H'(A) = A$.

The following version of the Abstract Positivstellensatz
is \cite[Theorem 1]{schw}:

\begin{theorem} \label{ks}  Given the preordered ring $(A,T)$ as
above and suppose $A = H'(A)$.
For any $a \in A$,  if $a >_P 0$ for all $P \in \spera$, then $a \in T$.
\end{theorem}

Consider the case where $A = \rr[X]$ and
$T = T_S$ for $S = \{ g_1, \dots , g_s \} \sse \rr[X]$.
Let $K = K_S$, then $K$ embeds
densely in $\spera$ and hence $H(A) = \{ f \in \rr[X] \mid
f$ is bounded on $S \}$.  If $S$ is compact, this implies
$H(A) = A$ and
Schm\"udgen's Theorem follows
from the following lemma \cite[Lemma 1]{bw}:

\begin{lemma} \label{wor}  With $A$, $T$, and $S$ as above, if $H(A) = A$
then $H'(A) = A$.
\end{lemma}

Our result follows from a generalization of Lemma \ref{wor},
which is \cite[Theorem 4.13]{schw}:

\begin{theorem} \label{sch}
Let $F$ be a subfield of $\rr$  and $(A,T)$ a preordered
$F$-algebra such that $F \sse H'(A)$ and $A$ has
finite transcendence degree over $F$.  Then
$$A = H(A) \Rightarrow A = H'(A).$$
\end{theorem}

We can now prove the existence of rational certificates of
positivity in Schm\"udgen's Theorem.  The argument is
exactly that of the proof of the general theorem above.

\begin{theorem}  Given $S = \{ g_1, \dots , g_s \} \sse \qq[X]$ and
suppose $K_S \sse \rr^n$ is compact.  Then for any and $f \in \qq[X]$ such
that  $f > 0$ on $K_S$, there is a representation of $f$ in the preordering
$T_S$,  $$f = \sum_{e \in \{0,1\}^s}  \sig_e g_1^{e_1} \dots g_s^{e_s},$$ with
all $\sig_e \in \sum \qq[X]^2$.
\end{theorem}

\begin{proof}  Let $T$ be the preordering in $\qq[X]$ generated
by $S$.  Since $K_S$ is compact, every element
of $\qq[X]$ is bounded on $K_S$.  Then $K_S$ dense in $\spera$ implies
that $H(\qq[X])= \qq[X]$, hence by Theorem \ref{sch} we have
$\qq[X] = H'(A)$.  Note that the condition $F \sse H'(A)$ holds
in this case since $\qq^+ = \sum \qq^2$.
The result follows from Theorem \ref{ks}.
\end{proof}

\section{Rational certificates for Putinar's Theorem}

Given $S = \{ g_1, \dots, g_s \}$, recall that the quadratic
module generated by $S$, $M_S$, is the set of
elements in the preordering $K_S$ with a ``linear"
representation, i.e.,
$$M_S = \{ \sig_0 + \sig_1 g_1 + \dots \sig_s g_s  \mid \sig_i \in \sum \rr[X]^2\}.$$
 In order to guarantee representations of positive
polynomials in the quadratic module, we need a condition stronger
than compactness of $K_S$, namely, we need $M_S$ to be archimedean.

The quadratic module $M_S$  is archimedean if
all elements of $\rr[X]$ are bounded by a positive integer
with respect to $M_S$, i.e., if for every $f \in \rr[X]$ there
is some $N \in \nn$ such that $N - f \in M_S$.  It is not too
hard to show that
$M_S$ is archimedean if
there is some $N \in \nn$ such that $N - \sum X_i^2 \in M_S$.
Clearly, if $M_S$ is archimedean, then $K_S$ is compact; the
polynomial $N - \sum X_i^2$ can be thought of as a ``certificate
of compactness".  However, the converse is not true, see
\cite[Example 6.3.1]{PD}.  The key to the algebraic proof of
Schm\"udgen's Theorem from the previous section is showing that in the
case of the preordering generated by a finite set of elements
from $\rr[X]$, the compactness of the semialgebraic set
implies that the corresponding preordering is archimedian.

In 1993, Putinar \cite{put} showed that
that if the quadratic module $M_S$ is archimedean, then we can replace the preordering
$T_S$ by the quadratic module $M_S$.

\begin{theorem} (Putinar)  Suppose that the quadratic module
$M_S$ is archimedean.  Then for every $f \in \rr[X]$ with
$f > 0$ on $K_S$, $f \in M_S$.
\end{theorem}

Lasserre's method for minimizing a polynomial on a compact semialgebraic
set, see \cite{LAS}, involves defining a sequence of semidefinite
programs corresponding to representations of bounded degree in $M_S$
whose solutions converge to the minimum.  In this context,  if $M_S$ is archimedean then
Putinar's Theorem implies the convergence of the semidefinite programs.
In practice, it is not clear how to decide if $M_S$
is archimedean for a given set of generators $S$, however in concrete
cases a polynomial $N - \sum X_i^2$ can be added to the generators
if an appropriate $N$ is known or can be computed.

Using an algorithmic proof of Putinar's Theorem due to M.~Schwieghofer \cite{schw2} we
can show that rational certificates exist for the theorem as long as we have
a polynomial $N - \sum X_i^2$ as one of our generators

\begin{theorem}  Suppose $S = \{ g_1, \dots , g_s \} \sse \qq[X]$ and
$N - \sum X_i^2 \in M_S$ for some $N \in \nn$.  Then given any $f \in \qq[X]$ such
that $f > 0$ on $K_S$, there exist
$\sig_0 \dots \sig_s, \sig \in \sum \qq[X]^2$ so that
$$f  = \sig_0 + \sig_1 g_1 + \dots + \sig_s g_s + \sig (N - \sum X_i^2).$$
\end{theorem}

\begin{proof}  The idea of Schweighofer's proof is to reduce to P\'olya's Theorem.
We follow the proof, making sure that each step preserves rationality.

Let $\Delta = \{ y \in [0,\infty)^{2n} \mid
y_1 + \dots + y_{2n} = 2n(N + \frac14) \} \sse \rr^{2n}$ and let $C$ be the compact
subset of $\rr^n$ defined by $C = l(\Delta)$, where $l : \rr^{2n} \rightarrow \rr^n$ is
defined by
$$y \mapsto \left(\frac{y_1 - y_{n+1}}{2}, \dots , \frac{y_n - y_{2n}}{2}\right)$$
Scaling the $g_i$'s by positive elements in $\qq$, we can assume that $g_i \leq 1$
on $C$ for all $i$.  The key to Schweighofer's proof is the
following observation \cite[Lemma 2.3]{schw2}:  There exists $\lambda \in \rr^+$
such that $q := f - \lambda \sum (g_i - 1)^{2k} g_i > 0$ on $C$.  Since we
can always replace $\lambda$ by a smaller value, we can assume $\lambda \in \qq$.

We have $q := f - \lambda \sum (g_i - 1)^{2k} g_i > 0$ on $C$, where $q \in \qq[X]$.
Write $q = \sum_{i=1}^d Q_i$, where $d = \deg q$ and $Q_i$ is the homogeneous part
of $q$ of degree $i$.  Let $Y = (Y_1, \dots , Y_{2n})$ and define in $\qq[Y]$
$$F(Y_1, \dots , Y_{2n})
:= \sum_{i=1}^d Q_i\left(\frac{Y_1 - Y_{n+1}}{2}, \dots , \frac{Y_n - Y_{2n}}{2}\right)
\left(\frac{Y_1 + \dots + Y_{2n}}{2n(N + 1/4)}\right)^{d-i}.$$  Then $F$ is homogenous
and $F > 0$ on $[0,\infty)^{2n} \setminus \{0\}$.  By P\'olya's Theorem, there is some
$k \in \nn$ so that $G := \left(\frac{Y_1 + \dots + Y_{2n}}{2n(N+1/4)}\right)^k F$ has
nonnegative coefficients as a polynomial in $\rr[Y]$.  Furthermore, since
$F \in \qq[Y_1, \dots , Y_{2n}]$, it is easy to see that $G \in \qq[Y]$.

Define $\phi : \qq[Y_1, \dots , Y_{2n}] \rightarrow \qq[X]$ by
$$\phi(Y_i) = N+\frac14 + X_i, \quad \phi(Y_{n+i}) = (N+\frac{1}{4}) - X_i, i=1, \dots , n$$
and note that $\phi(G) = q$ and $$\phi(Y_i) = (N+\frac{1}{4}) \pm X_i =
\sum_{j\neq i} X_j^2 + (X_i \pm \frac12)^2) + (N - \sum X_j^2) \in \sum \qq[X]^2 +
(N - \sum X_j^2).$$  Thus $\phi(G) = q$ implies there is a representation of
$q$ of the required type and then, since $f = q + \lambda  \sum (g_i - 1)^{2k} g_i$ with $\lambda \in
\qq$, we are done.
\end{proof}

\begin{remark}  In the preordering case (Schm\"udgen's Theorem), as noted above if
the semialgebraic set $K_S$ is compact, then it follows that the preordering $T_S$
in $\qq[X]$ is archimedean.   However it is more subtle in the quadratic module
case since it is not always clear how to decide if $M_S$ is archimedean for
a given set of generators $S$.  Thus an open question is the following:  Suppose $S \sse \qq[X]$
is a finite set of polynomials and $M_S$ is archimedean as a quadratic module
in $\rr[X]$.  Is it true that $M_S$ is archimedean as a quadratic module
in $\qq[X]$?  To put it more concretely, suppose $S = \{ g_1, \dots , g_s \} \sse \qq[X]$
and we know that there is some
$N \in \nn$ such that $$N - \sum X_i^2  = \sig_0 + \sig_1 g_1 + \dots + \sig_s g_s,$$
with $\sig_i \in \sum \rr[X]^2$.  Does there exist a representation with $\sig_i \in
\sum \qq[X]^2$?  Equivalently, does there exist $N \in \nn$ such that
for each $i = 1, \dots , n$ we can write $$N \pm X_i =
\sig_0 + \sig_1 g_1 + \dots + \sig_s g_s,$$ with $\sig_i \in \sum \qq[X]^2$?
\end{remark}

\bibliographystyle{amsplain}
\providecommand{\bysame}{\leavevmode\hbox to3em{\hrulefill}\thinspace}
\providecommand{\MR}{\relax\ifhmode\unskip\space\fi MR }
% \MRhref is called by the amsart/book/proc definition of \MR.
\providecommand{\MRhref}[2]{%
  \href{http://www.ams.org/mathscinet-getitem?mr=#1}{#2}
}
\providecommand{\href}[2]{#2}

\end{document}